\documentclass[11pt]{amsart}
\usepackage{amscd,amssymb,subfigure}
\usepackage[all]{xy}
\usepackage{graphicx}
\usepackage[colorlinks,plainpages,backref,urlcolor=blue]{hyperref}

\newtheorem{theorem}{Theorem}[section]
\newtheorem{corollary}[theorem]{Corollary}
\newtheorem{lemma}[theorem]{Lemma}
\newtheorem{prop}[theorem]{Proposition}

\theoremstyle{definition}

\newtheorem{example}[theorem]{Example}
\newtheorem{remark}[theorem]{Remark}

\newenvironment{romenum}
{ 

\begin{enumerate}}{\end{enumerate}}

\newcommand{\Z}{\mathbb{Z}}
\newcommand{\Q}{\mathbb{Q}}
\newcommand{\R}{\mathbb{R}}
\newcommand{\C}{\mathbb{C}}

\renewcommand{\k}{\Bbbk}
\newcommand{\RR}{{\mathcal R}}

\DeclareMathOperator{\im}{im}

\DeclareMathOperator{\id}{id}

\DeclareMathOperator{\Hom}{{Hom}}

\DeclareMathOperator{\inter}{{int}}

\newcommand{\surj}{\twoheadrightarrow}

\newcommand{\isom}{\xrightarrow{\,\simeq\,}}
\newcommand{\abs}[1]{\left| #1 \right|}

\begin{document}

\title[Algebraic monodromy and obstructions to formality]{%
Algebraic monodromy and obstructions to formality}

\author[Stefan Papadima]{Stefan Papadima$^1$}
\address{Institute of Mathematics Simion Stoilow, 
P.O. Box 1-764,
RO-014700 Bucharest, Romania}
\email{Stefan.Papadima@imar.ro}
\thanks{$^1$Partially supported by CNCSIS grant 
ID-1189/2009-2011 of the Romanian Ministry of 
Education and Research}

\author[Alexander~I.~Suciu]{Alexander~I.~Suciu$^2$}
\address{Department of Mathematics,
Northeastern University,
Boston, MA 02115, USA}
\email{a.suciu@neu.edu}
\urladdr{http://www.math.neu.edu/\~{}suciu}
\thanks{$^2$Partially supported by NSA grant H98230-09-1-0012, 
and an ENHANCE grant from Northeastern University}



\begin{abstract}
Given a fibration over the circle, we relate the eigenspace 
decomposition of the algebraic monodromy, the homological 
finiteness properties of the fiber, and the formality properties 
of the total space.  In the process, we prove a more 
general result about iterated group extensions. As an 
application, we obtain new criteria for formality of spaces, 
and $1$-formality of groups, illustrated by bundle 
constructions and various examples from  
low-dimensional topology and singularity theory.
\end{abstract}
\maketitle

\section{Introduction}
\label{sect:intro}

\subsection{Formality}
\label{subsec:intro1}
In his seminal paper on rational homotopy theory \cite{Su}, 
Sullivan  isolated an important topological property of spaces.  
A connected, finite-type CW-complex $X$ is said to be 
{\em formal}\/ if its minimal model is quasi-isomorphic to 
$(H^*(X,\Q),0)$; in particular, the rational homotopy type 
of $X$ is determined by its rational cohomology ring.  
Examples of formal spaces include compact 
K\"{a}hler manifolds \cite{DGMS} and complements of 
complex hyperplane arrangements \cite{Br}.    

There is a related property of groups.  A finitely generated 
group $G$ is said to be {\em $1$-formal}\/ if its Malcev completion 
(in the sense of Quillen \cite{Qu}) is a quadratic complete 
Lie algebra; see for instance \cite{PS04} for more details. 
If the first Betti number of $G$ is $0$ or $1$, then $G$ is $1$-formal.
Other examples of $1$-formal groups include finitely generated 
Artin groups \cite{KM}, finitely presented Bestvina-Brady 
groups \cite{PS-bb}, as well as pure welded braid groups \cite{BP}. 
As noted in \cite{Su}, fundamental groups of formal spaces 
are $1$-formal; thus, K\"{a}hler groups and arrangement 
groups belong to this class. In general, smooth complex 
quasi-projective varieties are not formal, but fundamental 
groups of complements of complex projective hypersurfaces 
are always $1$-formal \cite{Ko}.  

A well-known obstruction to formality is provided by 
the Massey products. If $X$ is a formal space, then 
all such products (of order $3$ and higher) in 
$H^*(X,\Q)$ vanish, up to indeterminacy. Similarly, 
if $G$ is a $1$-formal group, all higher-order Massey 
products in $H^2(G,\Q)$ vanish. 

\subsection{Algebraic monodromy}
\label{subsec:intro2}
In this note, we develop a new formality criterion, based 
on the interplay between three ingredients. 

The first ingredient is a close connection between 
the algebraic monodromy action on the homology of an 
infinite cyclic Galois cover, and the Aomoto--Betti numbers 
associated to the cohomology ring of the base.  This 
connection, established in \cite{PS-spectral} and adapted 
to our context in Section \ref{sect:mono}, requires no 
formality assumption. 

A second general fact, noted by Dwyer and Fried in 
\cite{DF} and further refined in \cite{PS-bns}, is that 
the homological finiteness properties (over $\C$) 
of a free abelian Galois cover are controlled by the 
jump loci for the twisted Betti numbers of the base, 
with rank $1$ complex coefficients. 

Finally, $1$-formality comes into play, by allowing an intimate 
relation between the Aomoto--Betti numbers and the twisted 
Betti numbers in degree one.  This relation is based on a key 
result from our joint work with A.~Dimca, \cite{DPS-jump}.

For a group $G$, denote by $b_i(G):=\dim_{\C} H_i(G,\C)$ 
its Betti numbers.  Our main result (proved in 
Section \ref{sect:proof thm}) reads as follows.

\begin{theorem}
\label{thm:main}
Suppose given two group extensions, 
\begin{align}
\xymatrixcolsep{14pt}
&\xymatrix{1\ar[r]& N\ar[r]& G \ar^{\nu}[r]& \Z \ar[r]&1}, 
\tag{$\xi_{\nu}$} \label{xinu}\\
&\xymatrix{1\ar[r]& K\ar[r]& \pi  \ar^{\chi}[r]&G \ar[r]& 1},
\tag{$\xi_{\chi}$}  \label{xichi}
\end{align}
such that
\begin{romenum}
\item \label{hyp1}
$b_1(N)<\infty$ and $b_1(K)<\infty$;
\item \label{hyp2}
$\pi$ is finitely presented and $1$-formal. 
\end{romenum}
Then, if the monodromy action of $1\in \Z$ on $H_1(N, \C)$ 
has eigenvalue $1$, the corresponding Jordan blocks are 
all of size $1$. 
\end{theorem}

The $1$-formality hypothesis is crucial here, even in the 
particular case (treated in Proposition \ref{prop=special}) 
when $\pi=G$ and $\chi=\id_G$.  For example, the 
Heisenberg group of integral, unipotent $3\times 3$ 
matrices is not $1$-formal, yet it arises as an 
extension of $\Z$ by $\Z^2$, with monodromy 
$\left( \begin{smallmatrix} 1 & 1\\ 0 & 1 
\end{smallmatrix} \right)$.  

\subsection{Bundles}
\label{subsec:intro3}

We illustrate Theorem \ref{thm:main} with a variety of 
geometric applications. To start with, we derive in 
Section \ref{sect:bundles} the following consequence.

\begin{theorem}
\label{thm:fgm intro}
Let $U$ be a closed, connected, smooth manifold, and let 
$\varphi\colon U \isom  U$ be a diffeomorphism, with mapping 
torus $U_{\varphi}$. Let $p\colon M\to U_{\varphi}$ be a locally 
trivial smooth fibration, with closed, connected fiber $F$. 
If the operator $\varphi_* \colon H_1(U, \C)\to H_1(U, \C)$ 
has eigenvalue $1$, with a Jordan block of size greater 
than $1$, then $M$ is not a formal space.
\end{theorem}

This theorem extends a result of Fern\'{a}ndez, Gray, and 
Morgan \cite{FGM}, valid for bundles with fiber $F=S^1$. 
It is unclear to us whether their approach---which relies on 
Massey products as an obstruction to formality---still 
works for a general fiber $F$.

One of the inspirations for this work was the following 
result of Eisenbud and Neumann \cite{EN}: 
If $L$ is a fibered graph multilink, then its 
algebraic monodromy has only $1\times 1$ Jordan 
blocks for the eigenvalue $1$. As noted by these 
authors, though, there do exist fibered links in $S^3$ 
for which larger Jordan blocks do occur. Necessarily, 
then, the complement of such a link is not formal.  
When combined with a result from \cite{PS-bns} on the 
Bieri-Neumann-Strebel invariant, our approach yields 
a formality criterion for closed, orientable $3$-manifolds 
which fiber over the circle, see Corollary \ref{cor:jordan blocks}. 

\subsection{Singularities}
\label{intro:org}

Given a reduced polynomial function $f\colon \C^2\to \C$ 
with $f(\mathbf{0})=0$ and connected generic fiber, 
there are two standard fibrations 
associated with it.  One is the Milnor fibration, 
$S^3_{\epsilon}\setminus K \to S^1$, where $K$ is the 
link at the origin, for $\epsilon >0$ small.   As shown 
by Durfee and Hain \cite{DH}, the link complement is formal. 
Theorem \ref{thm:main} allows us then to recover 
the basic fact that the algebraic monodromy 
has no Jordan blocks of size greater than $1$ for 
the eigenvalue $\lambda=1$. 

The other fibration, 
$f^{-1}(D_{\epsilon}^*) \to D_{\epsilon}^*$, is obtained by 
``base localization." As noted in \cite{ACD, D98}, the algebraic 
monodromy of this fibration can have larger Jordan blocks for 
$\lambda=1$, see Example \ref{ex:dimca}.  It follows that the 
space $f^{-1}(D_{\epsilon}^*)$ may be non-formal.  This 
observation is rather surprising, in view of the well-known 
fact that, upon localizing near $\mathbf{0}\in \C^2$ the 
complement of the curve $\{ f=0\}$, one obtains a space 
homotopy equivalent to the formal space 
$S^3_{\epsilon}\setminus K$. 

\section{Monodromy and the Aomoto complex}
\label{sect:mono}

\subsection{Modules over $\k\Z$}
\label{subsec:kz mod}

Fix a field $\k$.  The group-ring $\k\Z$ may be identified with the 
ring of Laurent polynomials in one variable, $\Lambda=\k[t^{\pm 1}]$. 
Since $\k$ is a field, the ring $\Lambda$ is a principal ideal 
domain.  Thus, every finitely-generated $\Lambda$-module 
$M$ decomposes into a finite direct sum of copies of $\Lambda$ 
(the free part), with torsion modules of the form $\Lambda/(f^i)$, 
where $f$ is an irreducible polynomial in $\Lambda$, and $i\ge 1$ 
is the size of the corresponding Jordan block. Define 
\begin{equation*}
\label{eq:fprimary}
M_f=\bigoplus_{i\ge 1} \big ( \Lambda/(f^i)\big )^{e_i(f)}
\end{equation*}
to be the {\em $f$-primary part} of $M$.  Clearly, the module 
$M_f$ consists of those elements of $M$ annihilated by some 
power of $f$. Hence, the following naturality property holds: 
Every $\Lambda$-linear surjection, $\phi\colon M \surj M'$, 
between finitely generated torsion $\Lambda$-modules 
restricts to a $\Lambda$-linear surjection, 
$\phi_f\colon M_f \surj M'_f$, for all irreducible polynomials 
$f\in \Lambda$.

\subsection{Homology of $\Z$-covers}
\label{subsec:zcov}

Let $X$ be a connected, finite-type CW-complex. Without loss 
of generality, we may assume $X$ has a single $0$-cell, 
say $x_0$. Let $G=\pi_1(X,x_0)$ be the fundamental group. 

Given an epimorphism $\nu\colon G\surj \Z$, let $X^{\nu}\to X$ 
be the  Galois cover corresponding to $\ker (\nu)$. Denote by 
$\k\Z_{\nu}$ the group ring of $\Z$, viewed as a (right) module 
over $\k{G}$ via extension of scalars:  $m\cdot g=m\nu(g)$, for 
$g\in G$ and $m\in \k\Z$.  By Shapiro's Lemma, 
$H_q(X^{\nu},\k)=H_q(X,\k\Z_{\nu})$, 
as modules over $\Lambda=\k\Z= \k [t^{\pm 1}]$. Since $X$ 
is of finite-type, $H_q(X,\k\Z_{\nu})$ is a finitely generated 
$\Lambda$-module, for each $q\ge 0$. 

Denote by $\nu_{\Z}$ the class in $H^1(X,\Z)$ determined 
by $\nu$. Let $\nu_{*}\colon H_1(X,\k)\to H_1(\Z,\k)=\k$ 
be the induced homomorphism. The corresponding 
cohomology class,  $\nu_{\k}\in H^1(X,\k)$, is the image 
of $\nu_{\Z}$ under the coefficient homomorphism $\Z\to \k$. 
We then have 
\begin{equation*}
\label{eq:nusquare}
\nu_{\k} \cup \nu_{\k} = 0\in  H^2(X,\k).
\end{equation*}
Indeed, by obstruction theory, there is a map 
$p\colon X\to S^1$ and a class $\omega\in H^1(S^1,\Z)$ 
such that $\nu_{\Z}=p^*(\omega)$. Hence, 
$\nu_{\Z}\cup\nu_{\Z}=p^*(\omega\cup\omega) =0$,  
and the claim follows by naturality of cup products 
with respect to coefficient homomorphisms. 

As a consequence, left-multiplication by $\nu_{\k}$ turns the 
cohomology ring $H^*( X,  \k)$ into a cochain complex, 
\begin{equation}
\label{eq:aomoto} \tag{*}
(H^*( X,  \k) , \cdot\nu_{\k})\colon 
\xymatrix{
H^0(X,\k) \ar^{\nu_{\k}}[r] & 
H^1(X,\k) \ar^{\nu_{\k}}[r] & 
H^2(X,\k) \ar[r] & \cdots 
}
\end{equation}
which we call the {\em Aomoto complex}\/ of $H^*(X, \k)$, with 
respect to $\nu_{\k}$. 
The {\em Aomoto--Betti numbers}\/ of $X$, with respect to 
the cohomology class $\nu_{\k}\in H^1(X,\k)$, are defined as 
\begin{equation*}
\label{eq:aomoto betti}
\beta_q(X, \nu_{\k}) := 
\dim_{\k} H^q(H^*( X,  \k) , \cdot\nu_{\k})\, .
\end{equation*}

In \cite{PS-spectral}, we defined a spectral sequence, 
starting at $E^{1}_{s,t} (X,\k\Z_{\nu})$, and converging to 
$H_{s+t}(X,\k\Z_{\nu})$. In the process, we investigated the 
monodromy action of $\Z$ on $H_{*}(X,\k\Z_{\nu})$, relating 
the triviality of the action to the exactness of the Aomoto 
complex defined by $\nu_{\k}\in H^1(X,\k)$, as follows. 

\begin{theorem}[\cite{PS-spectral}]
\label{thm:trivial action}
For each $k\ge 0$, the following are equivalent:
\begin{enumerate}
\item \label{m1}
For each $q\le k$, the $\k\Z$-module $H_q(X,\k\Z_{\nu})$ 
is finite-dimensional over $\k$, and contains no $(t-1)$-primary 
Jordan blocks of size greater than $1$.
\item \label{m2}
$\beta_0(X,\nu_{\k})=\cdots =\beta_k(X,\nu_{\k})=0$. 
\end{enumerate}
\end{theorem}

\subsection{Kernels of epimorphisms to $\Z$}
\label{subsec:zker}

Similar considerations apply to groups. Let $G$ be a group, and
let $\nu\colon G\surj \Z$ be an epimorphism with kernel $N$. 
Let $X=K(G,1)$ be a classifying space for $G$. Then the 
monodromy action of $\Lambda=\k[t^{\pm 1}]$ on 
$H_1(X, \k\Z_{\nu})= H_1(N, \k)$ is induced by 
conjugation:  the generator $t\in \Lambda$ acts on 
$H_1(N, \k)$ by $(\iota_g)_*$, where $g\in G$ is such 
that $\nu (g)=1$ and $\iota_g$ is the corresponding 
inner automorphism. 

Clearly, the Aomoto complex \eqref{eq:aomoto} depends 
in this case only on $G$ and $\nu_{\k}$; thus, we may 
denote $\beta_q(X, \nu_{\k})$ by $\beta_q(G, \nu_{\k})$.  
In the sequel, we will need the following group-theoretic 
version of Theorem \ref{thm:trivial action}.

\begin{corollary}
\label{cor=gmonotest}
Let $1\to N\to G  \xrightarrow{\,\nu\,} \Z \to 1$ be a group 
extension, with $G$ finitely presented. Then, the following 
are equivalent:
\begin{enumerate}
\item \label{mono1}
The  $\k\Z$-module $H_1(N, \k)$ is finite-dimensional 
over $\k$, and contains no $(t-1)$-primary Jordan block 
of size greater than $1$. 
\item  \label{mono2}
$\beta_1(G, \nu_{\k})=0$.
\end{enumerate}
\end{corollary}

\begin{proof}
Let $X$ be a finite presentation $2$-complex for $G$.  
Note that 
\[
H_0( X, \Lambda_{\nu})= \Lambda/(t-1) \text{ and }\,
\beta_0 ( X, \nu_{\k})=0.
\] 
Applying Theorem \ref{thm:trivial action} to $X$, for $k=1$, 
we find that $H_1( X, \Lambda_{\nu})$ is finite-dimensional 
over $\k$, with trivial monodromy action on the $(t-1)$-primary 
part, if and only if $\beta_1 ( X, \nu_{\k})=0$.  But 
$H_1( X, \Lambda_{\nu}) = H_1(N, \k)$ and 
$\beta_1 ( X, \nu_{\k})=\beta_1 ( G, \nu_{\k})$, 
and we are done. 
\end{proof}

\section{Monodromy and formality}
\label{sect:proof thm}

This section is devoted to proving Theorem \ref{thm:main} 
from the Introduction.  We first prove the special case 
in which $\pi=G$ and $\chi=\id_G$, and then reduce 
the general case to the special one, by means of  
two lemmas. 

\subsection{A special case}
\label{subsec:special}
We start by recalling a result (Corollary 6.7 
from \cite{PS-bns}), relating the homological 
finiteness properties of certain normal subgroups 
of a $1$-formal group to the vanishing of the 
corresponding Aomoto--Betti numbers. 

\begin{theorem}[\cite{PS-bns}]
\label{thm:b1ker}
Let $G$ be a finitely generated, $1$-formal group. 
Suppose $\nu\colon G\surj \Z$ is an epimorphism, 
with kernel $N$.  Then $b_1 (N) <\infty$ if and only if 
$\beta_1(G, \nu_{\C})=0$.
\end{theorem}

We are now ready to prove an important particular 
case of Theorem \ref{thm:main}.

\begin{prop}
\label{prop=special}
Let $G$ be a finitely presented, $1$-formal group, 
and let $\nu\colon G \surj \Z$ be an epimorphism,   
with kernel $N$. If $b_1(N)<\infty$, then the 
eigenvalue $1$ of the monodromy action of $1\in \Z$ on 
$H_1(N, \C)$ has only $1\times 1$ Jordan blocks.
\end{prop}

\begin{proof}
Since $G$ is $1$-formal and $b_1(N)< \infty$, 
Theorem \ref{thm:b1ker} insures that $\beta_1(G, \nu_{\C})=0$. 
The conclusion follows from Corollary \ref{cor=gmonotest}, 
for $\k=\C$. 
\end{proof}

\subsection{Two lemmas}
\label{subsec:lemmas}
Before proceeding, recall the extensions \eqref{xinu} and 
\eqref{xichi} from Section \ref{subsec:intro2}.  Set 
$\Gamma:=\ker (\nu\circ \chi)$. We then have an extension 
\[
\xymatrix{1\ar[r]& \Gamma \ar[r]& \pi \ar^{\nu\circ \chi}[r]& \Z \ar[r]&1}.  
\tag{$\xi_{\nu\chi}$} \label{xinuchi}
\]
Pulling back \eqref{xichi} to $N$ gives rise to another extension, 
\[
\xymatrix{1\ar[r]& K \ar[r]& \Gamma \ar^{\eta}[r]& N \ar[r]&1}. 
\tag{$\xi_{\eta}$} \label{xieta}
\]

The first lemma is a standard application of the Hochschild--Serre 
spectral sequence of a group extension, as recorded for instance 
in Brown \cite[p.~171]{B}. For the sake of completeness, 
we include a proof. 

\begin{lemma}
\label{lem=hs}
If both $b_1(N)$ and $b_1(K)$ are finite, so is $b_1(\Gamma)$.
\end{lemma}

\begin{proof}
Consider the spectral sequence of extension \eqref{xieta}, 
with $\Gamma$-trivial coefficients $\C$: 
\[
E^2_{st}= H_s(N, H_t(K, \C)) \Rightarrow H_{s+t}(\Gamma, \C) .
\]
Clearly, $E^2_{01}= H_1(K,\C)_N$ is a quotient of 
$H_1(K,\C)$, and $E^2_{10}= H_1(N,\C)$.  Thus, 
both $E^2_{01}$ and $E^2_{10}$ are finite-dimensional 
over $\C$.  Hence, $H_{1}(\Gamma, \C)$ is also 
finite-dimensional. 
\end{proof}

The second lemma shows that the Jordan blocks of 
size greater than $1$ for the eigenvalue $1$ propagate 
from \eqref{xinu} to \eqref{xinuchi}. As usual, write 
$\Lambda=\C\Z=\C[t^{\pm 1}]$. 

\begin{lemma}
\label{lem=propa}
If the $\Lambda$-module $H_1(\Gamma, \C)_{t-1}$ 
is trivial, then so is the module $H_1(N, \C)_{t-1}$.
\end{lemma}

\begin{proof}
Pick $g\in G$ such that $\nu (g)=1$, and then $x\in \pi$ such 
that $\chi(x)=g$.  Clearly, $\iota_g \circ \eta= \eta\circ \iota_x$. 
It follows that the surjection
\[
\eta_* \colon H_1(\Gamma, \C) 
\surj H_1(N, \C)
\]
is $\Lambda$-linear with respect to the corresponding 
monodromy actions. 

Now, since both $H_1(\Gamma, \C)$ and $H_1(N, \C)$ 
are finitely generated, torsion $\Lambda$-modules, we infer from 
the discussion in Section \ref{subsec:kz mod} that $\eta$ induces 
a $\Lambda$-linear surjection on $(t-1)$-primary parts, 
\[
(\eta_*)_{t-1} \colon H_1(\Gamma, \C)_{t-1} \surj 
H_1(N, \C)_{t-1}.
\]
The claim follows.
\end{proof}

\subsection{Proof of Theorem \ref{thm:main}}
\label{subsec:prooff main}

The finite-dimensionality assumption \eqref{hyp1},    
together with Lemma \ref{lem=hs} imply that $b_1(\Gamma)<\infty$.  

Suppose that the monodromy action on 
$H_1(N, \C)$ coming from extension \eqref{xinu} has 
eigenvalue $1$, with a Jordan block of size greater than $1$. 
By Lemma \ref{lem=propa}, the same holds for the monodromy action 
on $H_1(\Gamma, \C)$, coming from extension \eqref{xinuchi}. 

Finally, we use the formality assumption \eqref{hyp2}.   
From Proposition \ref{prop=special}, applied to extension \eqref{xinuchi}, 
we deduce that all Jordan blocks of the monodromy action 
on $H_1(\Gamma, \C)$ have size $1$.  But this is a contradiction, 
and the proof is finished.   
\hfill\qed

\section{Bundle constructions} 
\label{sect:bundles}

Let $U$ be a closed, connected, smooth manifold, and let 
$\varphi\colon U \isom  U$ be a diffeomorphism, with mapping 
torus $U_{\varphi}$. Let $p\colon M\to U_{\varphi}$ be a locally 
trivial smooth fibration, with closed, connected fiber $F$. 
Theorem \ref{thm:fgm intro} follows from the next proposition.

\begin{prop}
\label{prop:fgm1}
If the operator $\varphi_* \colon H_1(U, \C)\to H_1(U, \C)$ 
has eigenvalue $1$, with a Jordan block of size greater 
than $1$, then $\pi_1 (M)$ is not a $1$-formal group.
\end{prop}

\begin{proof}
The homotopy exact sequence of the mapping torus fibration, 
$U\hookrightarrow U_{\varphi} \to S^1$, yields an extension 
of fundamental groups, $1\to N\to G  \to \Z \to 1$.   
Clearly, $b_1(N)=b_1(U)< \infty$, and the monodromy 
action of $1\in \Z$ on $H_1(N,\C)$ coincides with $\varphi_*$.

Similarly, the fibration $F\hookrightarrow M \xrightarrow{\, p\, } 
U_{\varphi}$ yields an epimorphism on fundamental groups, 
$p_{\sharp} \colon\pi \surj G$. The group $K:= \ker (p_{\sharp})$ 
is a quotient of $\pi_1(F)$; hence, $b_1(K)< \infty$.

Plainly, the group $\pi=\pi_1(M)$ is finitely presented. 
Given the assumption on $\varphi_*$, Theorem \ref{thm:main}  
implies that $\pi$ is not $1$-formal.
\end{proof}

For a class $b\in H^*(U,\Z)$, denote by $b_{\R}$ its image in $H^*(U,\R)$. 
Proposition \ref{prop:fgm1} gives a substantial extension of the 
following result of Fern\'{a}ndez, Gray, and Morgan, contained 
in Proposition 9 and Note 10 from \cite{FGM}.  

\begin{theorem}[\cite{FGM}]
\label{thm:fgm}
Let $\varphi$ be an orientation-preserving diffeomorphism
of a closed, connected, oriented manifold $U$. Suppose 
there is an element $b\in H^1(U,\Z)$ with $\varphi^*(b)=b$, 
such that  $0\ne b_{\R}\in \im(\varphi^*-\id)$. 
If $E\to U_{\varphi}$ is a circle bundle, then $E$ is not formal. 
\end{theorem}

\begin{remark}
\label{rem:fgm}
When $U$ is a symplectic manifold, and $\varphi$ is a 
symplectic diffeomorphism, the circle bundle $E\to U_{\varphi}$ 
from above can be chosen so that $E$ carries a symplectic 
structure \cite[Theorem 4]{FGM}.   Yet, by Theorem \ref{thm:fgm} 
and the main result of \cite{DGMS}, the manifold $E$ 
admits no K\"{a}hler metric. 
\end{remark}

In a similar vein, the particular case of Theorem \ref{thm:main}
proved in Section \ref {sect:proof thm} has the following 
geometric counterpart.

\begin{corollary}
\label{prop:1jordan}
Let $F\hookrightarrow X\to S^1$ be a smooth fibration, 
with monodromy $h$.  Assume the fiber $F$ is connected and 
has the homotopy type of a CW-complex with finite $2$-skeleton.  
If the algebraic monodromy, $h_*\colon H_1(F, \C) \to H_1(F, \C)$, 
has eigenvalue $1$, with a Jordan block of size greater than $1$, 
then the group $G=\pi_1(X)$ is not $1$-formal.
\end{corollary}

\begin{proof}
Let $1\to N\to G  \to \Z \to 1$ be the extension of fundamental 
groups determined by the fibration.  From the hypothesis on 
the fiber $F$, the group $G$ is finitely presented, and 
$b_1(N)< \infty$.  The conclusion follows from 
Proposition \ref{prop=special}. 
\end{proof}

\section{Links and closed $3$-manifolds} 
\label{sect:3mfd}

\subsection{Links in the $3$-sphere}
\label{subsec:links}
The question of deciding whether a link complement is 
formal has a long history.  In fact, the higher-order products 
were defined by W.~Massey precisely for detecting higher-order 
linking phenomena for links such as the Borromean rings.  

The following example shows how the non-formality 
of a fibered link complement can be detected by the 
algebraic monodromy. It also shows that the $1$-formality 
hypothesis is essential in Theorem \ref{thm:main}.

\begin{figure}[t]
\centering
\includegraphics[height=1.8in,width=1.8in]{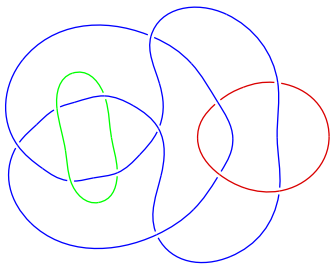} 
\caption{Eisenbud and Neumann's $3$-component link}
\label{fig:en link}
\end{figure}

\begin{example}
\label{ex:en}
Let $L$ be the $3$-component link in $S^3$ depicted in 
Figure \ref{fig:en link}.\footnote{%
This picture was drawn with the help of the Mathematica 
package {\sf KnotTheory}, by Dror Bar-Natan, and the 
graphical package {\sf Knotilus}, by Ortho Flint and Stuart Rankin.}   
As explained in \cite[Example 11.6]{EN}, 
the link exterior, $X=S^3 \setminus (L\times \inter D^2)$, fibers 
over $S^1$, with fiber $F$ a connected, compact surface, and 
with algebraic monodromy having two Jordan blocks of type 
$\left( \begin{smallmatrix} 1 & 1\\ 0 & 1 \end{smallmatrix} \right)$. 
By Corollary \ref{prop:1jordan}, the group $\pi_1(X)$ 
is not $1$-formal.
\end{example}

\subsection{$3$-manifolds fibering over the circle}
\label{subsec:3mfd}
Turning now to closed $3$-manifolds, let us first recall 
the following result from \cite{PS-bns}, based on the 
interplay between the Bieri--Neumann--Strebel invariant, 
$\Sigma^1(G)$, and the (first) resonance variety, 
$\RR^1(G):= \{ z\in \Hom (G, \R) \mid \beta_1(G, z)>0 \}$, 
of a $1$-formal group $G$. 

\begin{prop}[\cite{PS-bns}]
\label{prop:3mfd}
Let $M$ be a closed, orientable $3$-manifold which fibers 
over the circle.  If $b_1(M)$ is even, then $\pi_1(M)$ is 
not $1$-formal.  
\end{prop}

Combining this result with Corollary \ref{prop:1jordan} puts 
strong restrictions on the algebraic monodromy of a fibration 
$M^3\to S^1$.  

\begin{corollary}
\label{cor:jordan blocks}
Let $M$ be a closed, orientable $3$-manifold with 
$1$-formal fundamental group.  Suppose $M$ fibers 
over the circle, and the algebraic monodromy has $1$ 
as an eigenvalue.  Then, there are an even number of 
$1\times 1$ Jordan blocks for this eigenvalue, 
and no higher size Jordan blocks.
\end{corollary}

\begin{proof}
By Corollary \ref{prop:1jordan}, the algebraic monodromy 
has only $1\times 1$ Jordan blocks for the eigenvalue $1$. 
Let $m$ be the number of such blocks.   From the Wang 
sequence of the fibration, we deduce that $b_1(M)=m+1$. 
By Proposition \ref{prop:3mfd}, $m$ must be even. 
\end{proof}

\section{Singularities and fibrations} 
\label{sect:sing}

\subsection{Milnor fibration} 
\label{subsec:milnor fibration}
Let $f\colon (\C^{n+1},\mathbf{0})\to (\C,0)$ be a reduced
polynomial map. The link of $f$ at $\mathbf{0}$, denoted 
$K$, is the intersection of the hypersurface $\{f=0\}$ with the sphere 
$S^{2n+1}_{\epsilon} =\{x\in \C^{n+1} \mid \abs{x}=\epsilon \}$, 
for sufficiently small $\epsilon >0$. As shown by Milnor 
in \cite{Mi}, the map 
\begin{equation*}
\label{eq:milnor}
p\colon S^{2n+1}_{\epsilon} \setminus K \to S^1, \quad
p(x)=f(x)/\abs{f(x)}
\end{equation*}
is a smooth fibration. Moreover, the Milnor fiber $F=p^{-1}(1)$ 
has the homotopy type of an $n$-dimensional finite CW-complex.
Let $\varphi\colon F\to F$ be the monodromy map.  
As is well known, the monodromy operator, 
$\varphi_*\colon H_q(F,\C)\to H_q(F,\C)$, has only roots of 
unity as eigenvalues, with corresponding Jordan blocks of 
size at most $q+1$, see for instance \cite[p.~75]{D92}. 

In the case when $f$ has an isolated singularity at $\mathbf{0}$, 
more can be said. First of all, as shown by Milnor, the fiber $F$ 
is $(n-1)$-connected. Moreover, $\varphi_*\colon H_n(F,\C)\to H_n(F,\C)$ 
has Jordan blocks of size at most $n$ for the eigenvalue $1$, 
see for instance \cite{Eb}. 

Particularly tractable is the case of a reduced polynomial map 
$f\colon \C^2 \to \C$ with $f(\mathbf{0})=0$, when $\mathbf{0}$ 
is necessarily an isolated singularity.  As shown by Durfee 
and Hain in \cite[Theorem 4.2]{DH}, the link complement, 
$X=S^3_{\epsilon} \setminus K$, is a formal space. 
Corollary \ref{prop:1jordan}, then, explains from a 
purely topological point of view why the Jordan blocks 
of $\varphi_*\colon H_1(F,\C)\to H_1(F,\C)$ corresponding 
to the eigenvalue $1$ must all have size $1$. 

\subsection{Base localization} 
\label{subsec:base loc}
Let $f\colon \C^2 \to \C$ be a polynomial map 
with connected generic fiber. For $\epsilon >0$, there is a fibration 
$f\colon f^{-1}(D_{\epsilon}^*) \to D_{\epsilon}^*$ induced by $f$, 
where $D_{\epsilon}^*=\{t \in \C \mid 0<\abs{t}<\epsilon\}$; 
the topological type of this fibration is independent of 
small $\epsilon$. 

Now let $F$ be the generic fiber of $f$, and let $\varphi\colon F\to F$ 
be the monodromy of the fibration. Note that $F$ is a finite CW-complex, 
up to homotopy; see for instance \cite[p. 27]{D92}.  It was shown 
in \cite{ACD, D98} that $\varphi_*\colon H_1(F,\C)\to H_1(F,\C)$ 
can have a Jordan block of size $2$ for the eigenvalue $1$.  
The following concrete examples were communicated to us 
by Alex Dimca. 

\begin{example}
\label{ex:dimca}
Consider the polynomial $f=x^3+y^3+xy$. In this case, 
the central fiber $F_0=f^{-1}(0)$ has just a node singularity.
Using Theorem 1 or Corollary 6 from \cite{ACD}, we see that 
$\varphi_*\colon H_1(F,\C)\to H_1(F,\C)$ has a Jordan block of 
type $\left( \begin{smallmatrix} 1 & 1\\ 0 & 1 \end{smallmatrix} \right)$. 

Other examples are suitable translates of the polynomials 
$f_B$ and $f_C$ introduced in \cite[Example 4]{D98}, such that 
the central fiber corresponds to the Kodaira types $I_1$ or $I_2$. 
In this case, the central fibers are smooth, but the 
non-triviality of the monodromy is due to singularities at infinity. 
The polynomial $f_B$ (known as the Brian\c{c}on polynomial) 
is also considered in Example 5 from \cite{ACD}.
\end{example}

Further examples can be found in \cite{MW}. In all such examples, 
we conclude from Corollary \ref{prop:1jordan} that the group 
$\pi_1(f^{-1}(D_{\epsilon}^*))$ is not $1$-formal.

For a non-constant polynomial map, $f\colon \C^{2} \to \C$ 
with $f(\mathbf{0})=0$, let $V=\{f=0\}$ be the corresponding 
plane algebraic curve, and $Y=\C^2 \setminus V$ its complement. 
In recent work, Cogolludo--Agustin and Matei \cite{CM}, and 
Macinic \cite{Ma}, proved, by different methods, that $Y$ is 
a formal space.  Due to the local conic structure of affine 
varieties (see for instance \cite[p. 23]{D92}), for small enough 
$\delta>0$, the intersection 
\[
Y\cap \{ x\in \C^2 \mid \abs{x} \le \delta \}
\]
has the same homotopy type as the link complement 
$S^{3}_{\epsilon} \setminus K$.   In conclusion, the 
formality property of $Y$ is preserved by localization 
near $\mathbf{0} \in \C^2$, but it may be destroyed 
by base localization, that is, when replacing the 
complement $Y$ by 
\[
Y\cap f^{-1} (\{ t\in \C \mid \abs{t}< \epsilon \})= f^{-1}(D^*_{\epsilon}).
\]

\vspace{-2.8pc}
\newcommand{\arxiv}[1]
{\texttt{\href{http://arxiv.org/abs/#1}{arxiv:#1}}}
\renewcommand{\MR}[1]
{\href{http://www.ams.org/mathscinet-getitem?mr=#1}{MR#1}}


\begin{thebibliography}{00}

\bibitem{ACD} E.~Artal Bartolo, P.~Cassou-Nogu\`{e}s, A.~Dimca, 
{\em Sur la topologie des polyn\^{o}mes complexes}, in: 
Singularities (Oberwolfach, 1996), pp. 317--343, 
Progress in Math. vol.~162, Birkh\"{a}user, Basel, 1998.  
\MR{1652480} 

\bibitem{BP} B.~Berceanu, S.~Papadima,
{\em Universal representations of braid and braid-permutation 
groups}, J. Knot Theory Ramifications \textbf{18} (2009), 
no.~7, 999--1019.

\bibitem{Br} E.~Brieskorn,
{\em Sur les groupes de tresses}, in:
S\'{e}minaire Bourbaki, 1971/72, Lect. Notes in Math. 
\textbf{317}, Springer-Verlag, 1973, pp.~21--44. 
\MR{0422674}  

\bibitem{B}  K.~S.~Brown,
{\em Cohomology of groups}, Grad. Texts in Math., 
vol.~87, Springer-Verlag, New York-Berlin, 1982.
\MR{0672956}  

\bibitem{CM} J.~I.~Cogolludo-Agustin, D.~Matei,
{\em Cohomology algebra of plane curves, weak 
combinatorial type, and formality},
preprint \arxiv{0711.1951}.

\bibitem{DGMS}  P.~Deligne, P.~Griffiths, J.~Morgan, D.~Sullivan,
{\em Real homotopy theory of {K}\"{a}hler manifolds},
Invent. Math. \textbf{29} (1975), no.~3, 245--274.
\MR{0382702}  

\bibitem{D92} A.~Dimca, 
{\em Singularities and topology of hypersurfaces},
Universitext, Springer-Verlag, New York, 1992. 
\MR{1194180} 

\bibitem{D98}  A. Dimca,
{\em Monodromy at infinity for polynomials of two variables}, 
J. Algebraic Geom. \textbf{7} (1998), no.~4, 771--779.
\MR{1642749}  

\bibitem{DPS-jump} A.~Dimca, S.~Papadima, A.~Suciu,
{\em Topology and geometry of cohomology jump loci}, 
Duke Math. Journal \textbf{148} (2009), no.~3, 405--457.

\bibitem{DH}  A.~Durfee, R.~Hain,
{\em{Mixed {H}odge structures on the homotopy of links}}, 
Math. Ann. \textbf{280} (1988), no.~1, 69--83; 
\MR{0928298}   

\bibitem{DF} W.~G.~Dwyer, D.~Fried,
{\em Homology of free abelian covers. \textup{I}},
Bull. London Math. Soc. \textbf{19} (1987), no.~4, 350--352.
\MR{0887774}  

\bibitem{Eb} W.~Ebeling, 
{\em Monodromy}, in: Singularities and computer algebra, 
pp.~129--155, London Math. Soc. Lecture Note Ser., 
vol.~324, Cambridge Univ. Press, Cambridge, 2006. 
\MR{2228229} 

\bibitem{EN} D.~Eisenbud, W.~Neumann, 
{\em Three-dimensional link theory and invariants of plane curve
singularities}, Annals of Math. Studies, vol.~110, Princeton 
University Press, Princeton, NJ, 1985.
\MR{0817982} 

\bibitem{FGM} M.~Fern\'{a}ndez, A.~Gray, J.~Morgan, 
{\em Compact symplectic manifolds with free circle actions, 
and {M}assey products}, Michigan Math. J. \textbf{38} (1991), 
no.~2, 271--283. 
\MR{1098863}   

\bibitem{KM} M.~Kapovich, J.~Millson, 
{\em On representation varieties of {A}rtin groups, projective 
arrangements and the fundamental groups of smooth complex 
algebraic varieties}, Inst. Hautes \'{E}tudes Sci. Publ. Math.
\textbf{88} (1998), 5--95. 
\MR{1733326} 

\bibitem{Ko} T.~Kohno,
{\em On the holonomy {L}ie algebra and the nilpotent completion
of the fundamental group of the complement of hypersurfaces},
Nagoya Math. J. \textbf{92} (1983), 21--37.  
\MR{0726138} 

\bibitem{Ma}  A.~Macinic,
{\em Cohomology rings and formality properties of nilpotent 
groups}, preprint \arxiv{0801.4847}.

\bibitem{MW} F.~Michel, C.~Weber, 
{\em On the monodromies of a polynomial map from $\C^2$ 
to $\C$}, Topology \textbf{40} (2001), no.~6, 1217--1240. 
\MR{1867243} 

\bibitem{Mi} J.~W.~Milnor, 
{\em Singular points of complex hypersurfaces}, Ann. of Math. 
Studies, no.~61 Princeton Univ. Press, Princeton, N.J., 1968. 
\MR{0239612}  

\bibitem{PS04} S.~Papadima, A.~Suciu,
{\em Chen {L}ie algebras}, International Math.~Research 
Notices  \textbf{2004} (2004), no.~21, 1057--1086.  
\MR{2037049}  

\bibitem{PS-bb} S.~Papadima, A.~Suciu,
{\em Algebraic invariants for {B}estvina-{B}rady groups}, 
J. London Math. Society, \textbf{76} (2007), no.~2, 273--292.
\MR{2363416}  

\bibitem{PS-spectral} S.~Papadima, A.~Suciu,
{\em The spectral sequence of an equivariant chain 
complex and homology with local coefficients}, 
preprint \arxiv{0708.4262}, to appear in Trans. 
Amer. Math. Soc.

\bibitem{PS-bns} S.~Papadima, A.~Suciu,
{\em Bieri--{N}eumann--{S}trebel--{R}enz invariants and 
homology jumping loci}, preprint \arxiv{0812.2660}

\bibitem{Qu} D.~Quillen,
{\em Rational homotopy theory}, Ann. of Math.
\textbf{90} (1969), no.~2, 205--295.  
\MR{0258031}  

\bibitem{Su}  D.~Sullivan,
{\em Infinitesimal computations in topology},
Inst. Hautes \'{E}tudes Sci. Publ. Math.
\textbf{47} (1977), 269--331.
\MR{0646078}  

\end{thebibliography}
\end{document}